\documentclass{amsart}
\usepackage{amsmath,amssymb}
\usepackage{graphicx,tikz,mathtools}
\newtheorem{theorem}{Theorem}[section]
\newtheorem{proposition}[theorem]{Proposition}
\newtheorem{corollary}[theorem]{Corollary}

\newtheorem{lemma}[theorem]{Lemma}

\theoremstyle{definition}
\newtheorem{definition}[theorem]{Definition}

\theoremstyle{remark}
\newtheorem{remark}[theorem]{Remark}

\DeclareMathOperator\sech{sech}

\newcommand{\al}{\alpha}

\newcommand{\de}{\delta}
\newcommand{\ep}{\varepsilon}

\newcommand{\ga}{\gamma}
\newcommand{\ka}{\kappa}
\newcommand{\la}{\lambda}
\newcommand{\om}{\omega}
\newcommand{\si}{\sigma}

\newcommand{\vp}{\varphi}

\newcommand{\De}{\Delta}

\newcommand{\La}{\Lambda}
\newcommand{\Si}{\Sigma}
\newcommand{\Om}{\Omega}

\newcommand{\hPhi}{\widehat{\Phi}}

\newcommand{\hv}{\widehat{v}}

\def\RR{\mathbb{R}}

\def\ZZ{\mathbb{Z}}

\def\TT{\mathbb{T}}

\newcommand{\bE}{\mathbb{E}}
\newcommand{\bP}{\mathbb{P}}

\newcommand{\cA}{{\mathcal A}}
\newcommand{\cB}{{\mathcal B}}

\newcommand{\cF}{{\mathcal F}}
\newcommand{\cG}{{\mathcal G}}
\newcommand{\cH}{{\mathcal H}}

\newcommand{\cQ}{{\mathcal Q}}

\newcommand{\cT}{{\mathcal T}}

\newcommand{\cV}{{\mathcal V}}

\newcommand{\fS}{{\mathfrak S}}

\newcommand{\pd}{\partial}
\newcommand\minus\backslash

\newcommand\lan\langle
\newcommand\ran\rangle

\DeclareRobustCommand{\Bint}
{\mathop{%
		\text{%
			\settowidth{\intwidth}{$\int$}%
			\makebox[0pt][l]{\makebox[\intwidth]{$-$}}%
			$\int$}}}
\newcommand\Lone{\xrightarrow[\mathrm{a.s.}]{L^1}}

\newcommand\st{^{\mathrm{s}}}
\newcommand\un{^{\mathrm{u}}}

\renewcommand\leq\leqslant
\renewcommand\geq\geqslant
\newlength{\intwidth}

\newcommand\hphi{\widehat{\phi}}

\newcommand\tV{\widetilde{V}}

\numberwithin{equation}{section}

\begin{document}

\title[Non-integrability and chaos for natural Hamiltonian systems]{Non-integrability  and chaos for natural Hamiltonian systems with a random potential}

 \author{Alberto Enciso}
 \address{Instituto de Ciencias Matem\'aticas, Consejo Superior de
   Investigaciones Cient\'\i ficas, 28049 Madrid, Spain}
 \email{aenciso@icmat.es}

 \author{Daniel Peralta-Salas}
 \address{Instituto de Ciencias Matem\'aticas, Consejo Superior de
   Investigaciones Cient\'\i ficas, 28049 Madrid, Spain}
 \email{dperalta@icmat.es}

\author{\'Alvaro Romaniega}
\address{Instituto de Ciencias Matem\'aticas, Consejo Superior de
 Investigaciones Cient\'\i ficas, 28049 Madrid, Spain}
\email{alvaro.romaniega@icmat.es}

%%    General info
%\subjclass[2010]{35B38, 58J05, 58K45}
%\date{\today}
%
%\keywords{ }
%
\begin{abstract}
Consider the ensemble of Gaussian random potentials $\{V^L(q)\}_{L=1}^\infty$ on the $d$-dimensional torus where, essentially, $V^L(q)$ is a real-valued trigonometric polynomial of degree~$L$ whose coefficients are independent standard normal variables. Our main result ensures that, with a probability tending to~1 as $L\to\infty$, the dynamical system associated with the natural Hamiltonian function defined by this random potential, $H^L:=\frac12|p|^2+ V^L(q)$, exhibits a number of chaotic regions which coexist with a positive-volume set of invariant tori. In particular, these systems are typically neither integrable with non-degenerate first integrals nor ergodic. An analogous result for random natural Hamiltonian systems defined on the cotangent bundle of an arbitrary compact Riemannian manifold is presented too.
\end{abstract}

\maketitle

\section{Introduction}

It is intuitively clear that a typical Hamiltonian system should not be integrable. The reasons for this belief can be traced back at least to the early non-integrability results of Bruns and Poincar\'e in celestial mechanics~\cite{Whit} and to the realization that the integrals of motion are related to symmetries of the dynamical system. Making this intuition precise, however, is not easy. Indeed, the fact that generic Hamiltonian systems on the cotangent bundle $T^*M$ of a compact $d$-dimensional manifold~$M$ (or on a compact symplectic manifold) are not completely integrable in the sense of Liouville  was not established until the work of Markus and Meyer~\cite{MM74,MM80} in the 1970s. The gist of their proof is to show that generic Hamiltonian systems on~$T^*M$ do not admit nondegenerate action-angle coordinates anywhere on~$M$ because, by a deep transversality argument due to Robinson~\cite{Ro70}, they only have a countable number of degenerate periodic trajectories. More recently, Bessa, Ferreira, Rocha and Varandas~\cite{Bessa} have proved that, on any compact symplectic manifold, there exists a $C^2$-generic set of Hamiltonians  whose dynamics is topologically mixing on a full measure set of energy levels, which implies that the Hamiltonian system does not have any continuous first integrals other than the Hamiltonian.

As emphasized in~\cite{MM74}, these arguments do not carry over the important case of natural Hamiltonian systems, that is, functions on~$T^*M$ of the form
\begin{equation}\label{defHVintro}
H_V(q,p):=\frac12|p|^2_g + V(q)\,,
\end{equation}
where the kinetic term is given by a fixed Riemannian metric~$g$ on~$M$ and where the potential~$V$ is a smooth function on~$M$. The question of whether the Hamiltonian system defined by~\eqref{defHVintro} is non-integrable for a typical (in some sense) potential~$V$ remains wide open. The reason is that, as the kinetic part is fixed in this setting and one is only allowed to perturb the potential, which is independent of the momenta, transversality arguments cannot be used effectively.

Our objective in this paper is to establish a non-integrability result for natural Hamiltonian systems. Our approach is based on two main ideas. Firstly, since transversality does not seem to be well suited to deal with restricted classes of Hamiltonian dynamical systems, e.g.\ those with a fixed kinetic term as in~\eqref{defHVintro}, our approach will not rely on transversality techniques but on probabilistic methods. Thus, instead of statements that hold for a residual set of Hamiltonian systems, we will consider potentials that depend on a large number of independent, identically distributed  Gaussian random variables and prove statements that hold with a probability tending to~1 as said number of Gaussian variables tends to infinity. The second idea is that, instead of studying whether action-angle variables can exist or not, our goal is to show that typical natural Hamiltonian systems exhibit chaotic behavior (horseshoes). One this has been established, we infer the non-integrability property that the corresponding dynamical systems are not completely integrable with {nondegenerate} first integrals.

We recall that a Hamiltonian system on a $2d$-dimensional symplectic manifold is said to be {\em completely integrable}\/ if there exist $d$~smooth first integrals $I_1,\dots , I_d$ in involution that are functionally independent (i.e., the support of the $d$-form $dI_1\wedge \cdots \wedge d I_d$ is the whole manifold). We refer to~\cite[Definition~2.1]{Pater} for the general definition of when these first integrals are {\em nondegenerate}\/; for the purposes of this Introduction, it suffices to know that, in particular, they are nondegenerate when $I_1,\dots, I_d$ are Morse--Bott functions.

It is worth mentioning that Markus and Meyer did succeed in proving~\cite{MM74} that a natural Hamiltonian system of the form~\eqref{defHVintro} is not ergodic for a generic potential~$V$. This readily follows from their proof that, generically, there is a set of ergodic invariant tori of positive volume. We will also include a probabilistic analog of this result. Thus the picture which emerges from our result is that, typically, a random natural Hamiltonian system exhibits many chaotic regions coexisting with a positive-volume set of invariant tori.

Let us now introduce some definitions so we can rigorously state our main results. For concreteness, in this Introduction (and in most of the paper) we will restrict ourselves to the simplest case where $M$ is the $d$-dimensional torus $\TT^d:=(\RR/2\pi\ZZ)^d$ with the canonical flat metric. However, in Section~\ref{S.manifolds} we will consider the case of random Hamiltonian systems on the cotangent bundle of an arbitrary compact $d$-manifold. The definition of the random potentials is more cumbersome in this setting, but the same kind of arguments yield analogous, if slightly weaker, results. Details are given in Theorem~\ref{T2}.

Therefore, restricting our attention to the torus~$\TT^d$, we shall consider a trigonometric ensemble of Gaussian random potentials $\{V^L\}_{L=1}^\infty$, which we parametrize in terms of the degree~$L$ of
the trigonometric polynomial (see e.g.~\cite{Sodin}).
Specifically, let $\{a_n:n\in\ZZ^d\backslash\{0\}\}$ be complex-valued random variables whose real and
imaginary parts are centered normal variables {with variance 1/2} and which
are independent but for the constraint that $a_n=\overline{a_{-n}}$.  For $L\geq1$, we then define a random trigonometric potential of order~$L$ on~$\TT^d$ as
\begin{equation}\label{VL}
V^L(q):= \frac{1}{\sqrt{d_L}}\sum_{0<|n|_{\infty}\leq  L} {a_n} e^{in\cdot q}\,.
\end{equation}
Here $  |n|_{\infty}:=
  \max_{1\leq j\leq d}|n_j|$
denotes the largest component of the integer
vector $n\in\ZZ^d$ and
$$d_L\coloneqq\#\{n\in \ZZ^d:0<|n|_{\infty}\leq  L\}=(2L+1)^d-1$$
is a normalization constant. One should note that the constraint $a_n=\overline{a_{-n}}$ ensures that~$V^L$ is real valued, and that we have
removed the zero mode for convenience because adding a constant to the potential
does not change the associated Hamiltonian vector field.

To introduce the random natural Hamiltonian system defined by the above
potential, let us take the canonical flat metric on~$\TT^d$ and
consider, for each~$L\geq1$, the Hamiltonian function on~$\TT^d\times\RR^d$
\begin{equation}\label{pot_rando}
H^L(q,p):=\frac12|p|^2+ V^L(q)\,.
\end{equation}
The corresponding Hamiltonian vector field is
denoted by $X^L$.

For any integer~$N\geq1$, we will define an {\em $N$-grid}\/ on the torus
as the collection
\[
\cG^N:=\big\{Q_k^N: k\in\ZZ^d \cap [-N,N-1]^d\big\}
\]
of the $(2N)^d$ disjoint cubes of side $\pi/N$ given by
  \begin{equation}\label{eq:def Q}
  	Q^N_k:=\bigg(\frac{\pi k_1}N, \frac{\pi (k_1+1)}N\bigg)
    \times \cdots\times \bigg(\frac{\pi k_d}N, \frac{\pi (k_d+1)}N\bigg)\,.
  \end{equation}
Thus, for each positive integer~$N$,
\[
\bigcup_{Q\in \cG^N} \overline{Q}= \TT^d\,.
\]

Our main result then asserts that, with a probability that tends to~1 as $L\to\infty$, the random Hamiltonian vector field $X^L$ exhibits both chaos and a positive-volume set of invariant tori, so it is neither integrable with nondegenerate first integrals nor ergodic. In the statement, by a {\em horseshoe}\/ we mean, as usual, a compact hyperbolic invariant set on which the time-$T$ flow of the Hamiltonian field is conjugate to the shift map, for some $T>0$. In particular, this implies that the topological entropy of the field is positive.

\begin{theorem}\label{T1}
For any fixed $N\geq1$, the probability that the Gaussian random
Hamiltonian vector field $X^L$ has a horseshoe and a set of $d$-dimensional ergodic invariant tori of positive volume (of order $L^{-d}$) contained in the cotangent bundle of each of the cubes of the
$N$-grid~$\cG^N$ tends to~$1$ as $L\to\infty$. In both cases, the horseshoes and invariant tori we count are on the energy levels of a fixed nonempty interval $(h_1,h_2)$ which does not depend on $L$.

Furthermore, for $L\gg1$, the expected number of horseshoes and of the inner volume of $d$-dimensional ergodic invariant tori of~$X^L$ in this fixed energy levels are respectively greater than $\nu_* L^d$ and~$\nu_*$, where $\nu_*$ is a positive constant independent of~$L$.
\end{theorem}

For some insight into the expected number of horseshoes we compute in the energy interval $(h_1,h_2)$, note that, although the potential $V^L$ is obviously smooth because it is a trigonometric polynomial, typical frequencies in~$V^L$ are of order $O(L)$. Therefore, it makes sense that there are horseshoes and invariant tori of diameter $O(L^{-1})$, so the fact that there typically are $O(L^d)$ compact chaotic invariant sets is natural.

A straightforward corollary of this result is that, with a probability that tends to~1 as $L\to\infty$, the vector field~$X^L$ is not completely integrable with nondegenerate first integrals. When $d=2$, we also infer that the system does not admit any analytic first integral other than the Hamiltonian:

\begin{corollary}\label{cor}
With a probability tending to~$1$ as $L\to\infty$, the random Hamiltonian vector field $X^L$ is not completely integrable with nondegenerate first integrals. Moreover, when $d=2$, $X^L$ does not admit a nontrivial analytic first integral (that is, a function $I\in C^\omega(\TT^2\times\RR^2)$ invariant under the flow of~$X^L$ and such that the support of the $2$-form $dI\wedge dH^L$ is the whole space $\TT^2\times\RR^2$).
\end{corollary}

%In a way, this result could be understood as a probabilistic refinement of the
%elementary observation that, given any smooth potential $V_0$ on the
%torus (or on any other compact manifold) and any $\ep>0$, there
%exists another potential $V_1\in C^\infty(\TT^d)$ which approximates~$V_0$ in
%any fixed H\"older norm (that is, $\|V_0-V_1\|_{C^\al(\TT^d)}<\ep$
%with $\al\in(0,1)$) and the Hamiltonian vector field defined by
%$\frac12|p|^2_g+V_1$ exhibits a horseshoe that is contained in a cube of side~$\ep$. For the benefit of
%the reader, in Theorem~\ref{T.A} in the Appendix we provide a precise statement and a proof
%of this fact.

The article is organized as follows. In Section~\ref{S.Rd} we define a class of random potentials on $\RR^d$ and prove some useful properties of the probability measure they induce on the space of functions $C^r(\RR^d)$. The dynamics on $\RR^{2d}$ of the corresponding random Hamiltonian fields is studied in Section~\ref{S.dynamicsRd}, where we show that these fields exhibit horseshoes and a positive volume set of ergodic invariant tori almost surely. The connection between the random potentials on $\TT^d$ and the aforementioned random potentials on $\RR^d$ is established in Section~\ref{S.convergence}. Finally, in Section~\ref{S.conclusion} we prove Theorem~\ref{T1} and Corollary~\ref{cor}, while Section~\ref{S.manifolds} is devoted to an extension of our results to general compact manifolds.

\section{Definition of random potentials on~${\RR^d}$}
\label{S.Rd}

For $n\in\ZZ^d$, let $b_n$ be (real-valued) independent standard normal variables. It is then clear that
\begin{equation}\label{Eb}
	\bE(b_n\,b_{n'})=\de_{n,n'}
\end{equation}
for all $n,n'\in\ZZ^d$.

Consider the cube
\[
  \cQ:=[{-\La,\La}]^d\subset\RR^d\,,
\]
where
$\La$ is a large positive constant that will be fixed later
on. Note that
\begin{equation}\label{E.phin}
\phi_n(\xi):=e^{i \pi n\cdot  \xi/\La}\, 1_{\cQ}(\xi)\,,\qquad n\in\ZZ^d\,,
\end{equation}
is an orthonormal basis of $L^2(\cQ,(2\La)^{-d}\,d\xi)$, where $1_{\cQ}$ is the indicator function of the set~$\cQ$ and $(2\La)^{-d}\,d\xi$ is the normalized Lebesgue measure on~$\cQ$.

We now define the Gaussian random
function on~$\RR^d$
\begin{equation}\label{defW}
  W(x):= \sum_{n\in\ZZ^d} {b_n}\,\hphi_n(x)\,,
\end{equation}
where the real-valued function
\[
\hphi_n(x):=(2\La)^{-d}\,\int_{\cQ} e^{ix\cdot\xi} \,\phi_n(\xi)\,  d\xi
\]
is the Fourier transform of~$\phi_n$. Note that we have added a nonstandard normalization factor in the definition for future convenience.

It is standard~\cite{NS16,Sodin} that the {\em covariance kernel}\/ of the Gaussian
random function~$W$, defined as
\[
  \ka_W(x,x') := \bE[W(x)\,W(x')]\,,
\]
is $\ka_W(x,x')=\varkappa_W(x-x')$, where
\begin{equation}\label{eq:kernel}
	\varkappa_W(x):=\prod_{j=1}^d \frac{\sin \La x_j}{\La x_j}\,,
\end{equation}
denotes the Dirichlet kernel. The {\em spectral measure}\/ of~$W$ is therefore the normalized uniform measure on the cube~$\cQ$, that is,
\begin{equation}\label{E.KWint}
\varkappa_W(x)= (2\La)^{-d}\int_{\cQ}e^{ix\cdot\xi}\, d\xi\,.
\end{equation}

The following proposition shows that $W$ is, almost surely, an analytic, polynomially bounded function on $\RR^d$:

\begin{proposition}\label{P.conv}
With probability~$1$, $W$ is an analytic function on~$\RR^d$ and
\[
(1+|x|^2)^{-\rho/2}\, W(x)
\]
is bounded, where $\rho$ is any fixed constant larger than~$d$. Furthermore, for any~$r\geq0$, the
series~\eqref{defW} converges locally in~$C^r$ with probability~$1$.
\end{proposition}

\begin{proof}
As the covariance kernel $\varkappa(x)$ is non-negative definite in the sense that it is the Fourier transform of a nonnegative Borel measure, and
\[
\sup_{x\in\RR^d} |\pd^\al \varkappa_W(x)|<\infty
\]
for any multiindex~$\kappa$, a theorem of Kolmogorov~\cite[A.9]{NS16} ensures that the random field $W$ is of class $C^\infty$ almost surely.

Our next goal is to show that the series~\eqref{defW} converges locally uniformly in~$C^r$. Replacing the variables $(x,\xi)$ by $(\La x/\pi,\pi \xi/\La)$ if necessary, we can assume without any loss of generality that $\La=\pi$. The functions $\phi_n$ and $\hphi_n$ can then be written as
\[
\phi_n(\xi)= \prod_{j=1}^d \Phi_{n_j}(\xi_j)\,,
\]
where $\Phi_m\in L^2(\RR)$ is defined for $m\in\ZZ$ as
\[
\Phi_m(\eta):= e^{im\eta}\,1_{[-\pi,\pi]}(\eta)\,.
\]
Let us now observe that the Fourier transform of~$\Phi_m$ can be written in terms of the sinc function
\[
f(t):= \frac{\sin \pi t}{\pi t}
\]
as
\[
\hPhi_m(y)=  f(m+y)\,,
\]
which immediately yields the bound
\begin{equation}\label{E.pdyphin}
|\pd_y^k \hPhi_m(y)|\leq C_k [1 +(m+y)^2]^{-1/2}\,.
\end{equation}
Here the constant $C_k$ is independent of~$m$.
Plugging this formula into the expression for $\hphi_n$, we obtain that
\[
\hphi_n(x)=\prod_{j=1}^d f(n_j+x_j)\,,
\]
so
\begin{equation}\label{E.boundhPhi}
	|\pd^\al \hphi_n(x)|\lesssim \prod_{j=1}^d[1+ (n_j+x_j)^2]^{-1/2}
\end{equation}
for any multiindex~$\al$. The implicit constants only depend on~$\al$ and~$d$.

Consider now the partial sums
\[
S_N(x):= \sum\limits_{\substack{|n_j|\leq N \\ 1\leq j\leq d}} {b_n}\,\hphi_n(x)\,.
\]
For each fixed $x\in\RR^d$ and any fixed multiindex~$\al$, one can then use these bounds and the identity~\eqref{Eb} to write, for each $1\ll N< N'$,
\begin{align*}
\bE\big[ \pd^\al S_{N'}(x)-\pd^\al S_N(x)\big]^2 &= \bE  \left[\sum\limits_{\substack{N<|n_j|\leq N' \\ 1\leq j\leq d}}  b_n \pd^\al \hphi_n(x)\right]^2= \sum\limits_{\substack{N<|n_j|\leq N' \\ 1\leq j\leq d}}  [\pd^\al\hphi_n(x)]^2 \\
&\lesssim \prod_{j=1}^d \sum\limits_{N<|n_j|\leq N' } n_j^{-2}\lesssim  N^{-d}\,.
\end{align*}
Here we have used that, for a fixed~$x$ and any $n_j$ large enough,
\[
1 +(n_j+x)^2\geq C_x n_j^2
\]
with a constant that depends on~$x$ but not on~$n_j$. Thus $\pd^\al S_N(x)$ is a Cauchy sequence and converges to $\pd^\al W(x)$ in quadratic mean (and therefore in probability). The Ito--Nisio theorem~\cite[A.6]{NS16} then ensures that the random field $S_N$ converges to $W$ locally in~$C^r$ with probability~1, for any $r\geq 0$.

Finally we need to estimate the growth of $W$ at infinity. As the series converges locally in~$C^r$, let us now fix any constant $R>0$ and write the expectation value of the $L^2$~norm of~$\pd^\al W$ on the ball~$B_R$ as
\begin{align*}
\bE\|\pd^\al W\|_{L^2(B_R)}^2 &= \sum_{n,n'\in\ZZ^d} \bE(b_n {b_{n'}}) \int_{B_R} \pd^\al\hphi_n(x)\,\pd^\al \hphi_{n'}(x)\, dx\\
&=\sum_{n\in\ZZ^d} \int_{B_R} |\pd^\al\hphi_n(x)|^2\, dx\,,
\end{align*}
where to pass to the second line we have used the identity~\eqref{Eb}

The bound~\eqref{E.boundhPhi} immediately yields that
\[
\int_{B_R} |\pd^\al\hphi_n(x)|^2\, dx \lesssim \prod_{j=1}^d\int_{-R}^R\frac{dx_j}{1+(n_j+x_j)^2}\lesssim R^d\prod_{j=1}^d m(R,n_j)\,,
\]
where the function $m(R,\nu)$ is defined as
\[
m(R,\nu):=\left\{
            \begin{array}{ll}
              1\,, &  |\nu|\leq 2R\,,\\
              \nu^{-2}\,, & |\nu|>2R\,.
            \end{array}
          \right.
\]
To obtain the last inequality it suffices to note that $x\in B_R$ and, if $|n_j|> 2R$,
\[
|n_j+x_j|>\frac{|n_j|}2\,.
\]
Accordingly, we can write
\begin{align*}
\sum_{n\in \ZZ^d} \int_{B_R} |\pd^\al\hphi_n(x)|^2\, dx	&\lesssim R^d\sum_{n\in\ZZ^d}\prod_{j=1}^d m(R,n_j)\\
	&= R^d \prod_{j=1}^d\Big(\sum_{n_j\in\ZZ}m(R,n_j)\Big)=R^d\Big(\sum_{\nu\in\ZZ}m(R,\nu)\Big)^d\\
&\lesssim R^{2d}\,,
\end{align*}
where we have used the obvious bound
\[
\sum_{\nu\in\ZZ}m(R,\nu)=\sum_{|\nu|\leq 2R}1+\sum_{|\nu|>2R}\nu^{-2}\lesssim R\,.
\]
Therefore, the expectation of the Sobolev norm of~$W$ on~$B_R$ is bounded as
\[
\bE\|W\|_{H^s(B_R)}^2\lesssim R^{2d}
\]
for any $s\geq0$ and any $R>0$ (the implicit constant only depends on $s$).

Moreover, for any constant $\rho>d$, one can use this bound to compute the expectation of the Sobolev norm of $\langle x\rangle^{-\rho}W$ using a dyadic collection of concentric balls as
\begin{align*}
\bE\|\langle x\rangle^{-\rho}W\|_{H^s(\RR^d)}^2 & = \bE\|\langle x\rangle^{-\rho}W\|_{H^s(B_2)}^2 + \sum_{N=1}^\infty \bE\|\langle x\rangle^{-\rho}W\|_{H^s(B_{2^{N+1}}\backslash B_{2^N})}^2\\
&\lesssim \bE\|W\|_{H^s(B_2)}^2+ \sum_{N=1}^\infty 2^{-2N\rho}\,\bE\|W\|_{H^s(B_{2^{N+1}}\backslash B_{2^N})}^2\\
&\lesssim \bE\|W\|_{H^s(B_2)}^2+ \sum_{N=1}^\infty 2^{-2N\rho}\,\bE\|W\|_{H^s(B_{2^{N+1}})}^2\\
&\lesssim 1+ \sum_{N=1}^\infty 2^{-2N\rho} \,2^{2N d}\lesssim1\,.
\end{align*}
By Sobolev's inequality, that is
\[
\|W\|_{C^r(\RR^d)}\lesssim \|W\|_{H^s(\RR^d)}
\]
with $s>r+d/2$, we obtain the desired pointwise bound.

Also, since the Fourier transforms $\widehat S_N$ of each partial sum $S_N$ are distributions supported in~$\cQ$, it is standard that the polynomial bound and the convergence result we have established ensure that $\widehat S_N\to \widehat W$ in the sense of tempered distributions as $N\to\infty$. In particular, the Fourier transform of~$W$ is supported in~$\cQ$ almost surely, so $W$ is an analytic function almost surely by the Payley--Wiener theorem.
\end{proof}

Since the Gaussian field~$W$ is smooth with
probability~1 by Proposition~\ref{P.conv}, it is standard that it defines a Gaussian probability
measure~$\mu_W$ on the space of~$C^r$~functions on~$\RR^d$,
 for any fixed integer~$r$.
This space is endowed with its usual Borel $\si$-algebra~$\fS$,
which is the minimal $\si$-algebra containing the intervals
\[
  I(x,a,b):=
\{w\in C^r(\RR^d): w(x)\in[a,b)\}
\]
for all $x,a,b,\in\RR$ with $a<b$. To spell out the details, let us denote
by~$\Omega$ the sample space of the random variables $b_n$ and show
that the random function~$W$ is a measurable map from~$\Omega$ to~$C^r(\RR^d)$. Since the
$\si$-algebra of~$C^r(\RR^d)$ is generated by point evaluations,
it suffices to show that $W(x)$ is a measurable function
$\Omega\to\RR$ for each~$x\in\RR^d$. But this is obvious because $W(x)$ is
the limit of finite linear combinations
of the random variables~$b_n$, which are of course measurable. In what follows, we will not
mention the $\si$-algebra explicitly to keep the notation simple. Also, in view of later applications, we  henceforth assume that~$r\geq2$.

In the following lemma we establish some basic properties of the probability measure~$\mu_W$. We recall that~$\mu_W$ is said to be {\em translationally invariant}\/ if $\mu_W(\tau_y \cA)=\mu_W(\cA)$
for all measurable sets~$\cA\subset\fS$ and all~$y\in \RR^d$. Here $\tau_y$~denotes
the translation operator on $C^r$~functions, defined as
\[
\tau_y w(x):=
w(x+y)\,.
\]

\begin{lemma}\label{L.ergodic}
The probability measure $\mu_W$ is translationally invariant and ergodic with respect to the action of translations $\{\tau_y:y\in\RR^d\}$.  Furthermore, if~$\Phi$ is an $L^1$~random variable on the probability space $(C^r(\RR^d),\fS,\mu_W)$, then
\[
\lim_{R\to\infty}\Bint_{B_R} \Phi\circ\tau_y\, dy = \bE\Phi
\]
both $\mu_W$-almost surely and in $L^1(C^r(\RR^d),\mu_W)$.
\end{lemma}
\begin{proof}
Since the covariance kernel $\ka_W(x,x')$ only depends on~$x-x'$, the
probability measure~$\mu_W$ is obviously translationally invariant. Also, note
	that $(y,w)\mapsto \tau_yw$ defines a continuous map
	\[
	\RR^d\times C^r(\RR^d)\to C^r(\RR^d)\,,
	\]
	so the map $(y,w)\mapsto \Phi(\tau_yw)$ is measurable on the product
	space $\RR^d \times C^r(\RR^d)$. Wiener's ergodic
	theorem~\cite{NS16, Bec81} then ensures that, for $\Phi$ as in the
	statement, there is a random variable $\Phi^*\in
	L^1(C^r(\RR^d),\mu_W)$ such that
	\[
	\Bint_{B_R} \Phi\circ\tau_y\, dy \Lone \Phi^*
	\]
	as $R\to\infty$. Furthermore, $\Phi^*$ is translationally invariant
	(i.e., $\Phi^*\circ\tau_y=\Phi^*$ for all $y\in\RR^d$ almost surely) and $\bE\Phi^* =
	\bE\Phi$.
	
	Also, as the spectral measure of \eqref{eq:kernel} has no atoms, a theorem of Grenander, Fomin and Maruyama (see
	e.g.~\cite[Appendix B]{NS16}) ensures that the action of the translations
	$\{\tau_y: y\in\RR^d\}$ on the probability space
	$(C^r(\RR^d),\fS,\mu_W)$ is ergodic.
	As the measurable
	function~$\Phi^*$ is translationally invariant, one then infers that
	$\Phi^*$ is constant $\mu_W$-almost surely. As $\Phi$ and $\Phi^*$ have
	the same expectation, then $\Phi^*=\bE \Phi$ almost surely.
	The proposition then follows.
\end{proof}

Let $L^2_\cH(\cQ)$ denote the space of Hermitian square-integrable functions supported on~$\cQ$ (that is, those for which $f(x) = \overline{f(-x)}$) and let
\[
\cF_\cQ:= \overline{\{w\in C^r(\RR^d): \widehat w\in L^2_\cH(\cQ)\}}
\]
be the closure of the space of $C^r$-smooth functions whose Fourier transform is in~$L^2_\cH(Q)$. This closure is taken with respect to the Fr\'echet topology on~$C^r$, corresponding to uniform $C^r$-convergence on compact sets.

\begin{remark}\label{R.obvious}
Any function $w$ on $\RR^d$ which is the Fourier transform of a tempered distribution $\widehat w$ supported on $\cQ$ is in $\cF_\cQ$. This follows from the fact that $\widehat w$ can be approximated, in the sense of tempered distributions, by functions in $L^2(\cQ)$, which yields $C^r$-convergence in compact sets when the Fourier transform is taken.

\end{remark}

To conclude this section, we present a variation of a standard lemma (see, e.g., \cite[Appendix A.7]{NS16})  adapted to our setting. To state this result, recall that the {\em support}\/ of the probability measure $\mu_W$ is the closed set consisting of those $v\in C^r(\RR^d)$ such that $\mu_W(U)>0$  for any open neighborhood $U \ni v$.

\begin{lemma}\label{L.support}
The support of the probability measure~$\mu_W$ is~$\cF_\cQ$. In
particular, given any $w\in\cF_\cQ$, any compact $K\subset\RR^d$ and
any $\ep>0$,
\[
\mu_W\big(\{ v\in C^r(\RR^d): \|v-w\|_{C^r(K)}<\ep\}\big)>0\,.
\]

\end{lemma}
	\begin{proof}
		First, we show that for all~$v$ in
		\[
		\cF_\cQ'\coloneqq\{w\in C^r(\RR^d): \widehat w\in L^2_\cH(\cQ)\}\,,
		\]
		every compact set $K\subset \RR^d$, and every $\varepsilon>0$,  the probability
		\[
		\bP\{ \| W-v \|_{C^r(K)} < \varepsilon \}
		\]
		is positive.
		
		Since the Fourier transform $\hv$ is an $L^2$~function supported on~$\cQ$, one can write
		\begin{equation}\label{E.hv}
		\hv (\xi)=\sum_{n\in\ZZ^d} c_n\, \phi_n(\xi)
		\end{equation}
		where the series converges in~$L^2(\RR^d)$ and where the functions $\phi_n$ were defined in~\eqref{E.phin}. Note that
		\[
		\|\hv\|_{L^2}^2= \sum_{n\in\ZZ^d}|c_n|^2\,.
		\]
				For future reference, let us introduce the notation
		\[
		M_N:=\#\{n\in\ZZ^d: |n|\leq N\}
		\]
		and observe that
		\[
		m_0:=\sup_{n\in\ZZ^d}\|\hphi_n\|_{C^r(\RR^d)}
		\]
		is finite by the estimate~\eqref{E.boundhPhi}. In particular, since
		\[
		|\pd^\al v(x)|=(2\La)^{-d}\left|\int_{\cQ} (i\xi)^\al \hv(\xi)\, e^{ix\cdot\xi}\, d\xi\right|\lesssim \|\hv\|_{L^2}
		\]
		for any multiindex~$\al$ and any $x\in\RR^d$, it follows from the fact that that the expansion~\eqref{E.hv} converges in~$L^2$ that, for any $\ep>0$, there exists some $N_\ep$ such that
		\[
		v_{N}(x):= \sum_{|n|>N} c_n\, \hphi_n(x)
		\]
		satisfies
		\[
		\|v_{N}\|_{C^r(K)}<\ep
		\]
		for all $N\geq N_\ep$.
		
		Likewise, as the series
		\[
		W(x)=\sum_{n\in\ZZ^d} b_n\,\hphi_n(x)
		\]
		converges in $C^r(K)$ with probability~1, one deduces that for every~$\ep>0$ there exists some $N'_\ep$ such that the random field
		\[
		W_{N}(x):= \sum_{|n|>N} b_n\,\hphi_n(x)
		\]
		satisfies
		\[
		\bP\{ \|W_{N}\|_{C^r(K)}<\ep\}>0
		\]
		for all $N\geq N_\ep'$.
		
		Let us now pick $N:= \max\{N_{\ep/4},N_{\ep/4}'\}$. The elementary inequality
		\[
		\|W-v\|_{C^r(K)}\leq m_0\sum_{|n|\leq N}|b_n-c_n|+ \|W_N\|_{C^r(K)}+\|v_N\|_{C^r(K)}
		\]
		then permits to conclude that
		\begin{align*}
	\bP\{\| W-v \|_{C^r(K)} < \varepsilon \} &\geq \bP\{\| W_N \|_{C^r(K)} <\ep/4\}\\ &\cdot \bP\{\| v_N \|_{C^r(K)} <\ep/4\}
	\cdot \prod_{|n|\leq N}P(c_n,\ep/(2m_0M_N))\\
&=\bP\{\| W_N \|_{C^r(K)} <\ep/4\}\cdot \prod_{|n|\leq N}P(c_n,\ep/(2m_0M_N))\,,
	\end{align*}
where
\[
	P(c,\rho):=\int_{c-\rho}^{c+\rho}\frac{e^{-t^2/2}}{\sqrt{2\pi}}dt
\]
		is the probability that a standard normal random variable~$Z$ satisfies $|Z-c|<\rho$. Thus
		\[
		\bP\{\| W-v \|_{C^r(K)}<\ep\}>0\,.
		\]
		As the support of a measure is a closed set with respect to the Fr\'echet topology on~$C^r(\RR^d)$, this implies that
		$		 \cF_\cQ=\overline{ \cF_\cQ'}$ must be contained in the support of~$\mu_W$.

		Let us now prove the converse inclusion. To this end, pick some $v\in C^r(\RR^d)$ in the support of the measure, that is, such that
		\[
		\de:= \mu_W(\{w\in C^r(\RR^d): \|w-v\|_{C^r(K)}<\tfrac\ep2\})= \bP\{\|W-v\|_{C^r(K)}<\tfrac\ep2\}>0
		\]
	for any $\ep>0$. Our objective is then to show that
		\[
		\mu_W(\{w\in\cF_\cQ : \|w-v\|_{C^r(K)}<\ep\})
		\]
		is also positive.
		
		As the series for~$W$ converges in~$C^r(K)$ by Proposition~\ref{P.conv}, there is some $N$ such that
		\[
		\bP\{\|W_N\|_{C^r(K)}>\tfrac\ep2\}<\tfrac\de2\,.
		\]
		As
		\[
		W_<(x):=\sum_{|n|\leq N} b_n\hphi_n(x)
		\]
		obviously is a function in~$\cF_\cQ$, the inequality
		\[
		\|W_<-v\|_{C^r(K)}\leq \|W-v\|_{C^r(K)}+ \|W_N\|_{C^r(K)}
		\]
		implies that
		\begin{multline*}
		\mu_W(\{	w\in\cF_\cQ : \|w-v\|_{C^r(K)}<\ep\}) \geq \bP\{ \|W_< - v\|_{C^r(K)}<\ep\}\\
		\geq \bP\{\|W-v\|_{C^r(K)}<\tfrac\ep2\}-\bP\{\|W_N\|_{C^r(K)}>\tfrac\ep2\}>\tfrac\de2\,,
		\end{multline*}
		for any $\ep>0$. Therefore, we conclude that the support of ~$\mu_W$ is contained in $\cF_\cQ$ and the proposition follows.
	\end{proof}

\section{Dynamics of random Hamiltonian systems on~$\RR^d$}
\label{S.dynamicsRd}

In this section we shall consider the dynamics of the Gaussian
random vector field
\[
X_W:=P\,\frac\pd{\pd x}-\nabla W(x)\,\frac\pd{\pd P}
\]
associated with the random potential~$W$ which we defined in~\eqref{defW}. More generally, given any potential $w\in C^r(\RR^d)$, throughout this section we use the notation $X_w$  for the vector field corresponding to the Hamiltonian function
\[
H_w(x,P):=\frac12|P|^2+ w(x)\,,
\]
with $(x,P)\in\RR^{2d}$ being canonically conjugated variables. Through~$X_W$, the probability measure~$\mu_W$
induces a probability measure~$\mu_X$ on the space of $C^{r-1}$~vector fields
on~$\RR^{2d}$.

As mentioned in the Introduction, by a chaotic invariant set we mean a horseshoe, which is a
a connected compact hyperbolic set invariant under the flow of the
Hamiltonian field and which has positive topological entropy. If $\gamma$ is a hyperbolic periodic orbit of the field, we shall denote its stable and unstable invariant manifolds by $W\st(\gamma)$ and $W\un(\gamma)$, respectively.

The following proposition shows the existence of an analytic potential on $\RR^d$ whose corresponding Hamiltonian field is chaotic and has a positive volume set of invariant tori, and such both properties are preserved by small perturbations.

\begin{proposition}\label{L.chaos0}
In the set
\[
\Om:=\{(x,P)\in\RR^{2d}:|x|^2+|P|^2<25\}
\]
and for a nonempty interval of energy values $(h_1,h_2)$,
the Hamiltonian vector field associated with the potential
\[
W_\eta(x):=\frac12 x_1^2+\sum_{j=2}^d(1-\cos x_j)+\frac{\eta}{2}\sum_{j=2}^d(x_1-x_j)^2\,,
\]
with $\eta>0$ a suitably small constant, exhibits:
\begin{enumerate}
   \item Two hyperbolic periodic orbits $\gamma_+$ and $\gamma_-$ with transverse homoclinic intersections. In particular, this implies the existence of horseshoes.

   \item A set of $d$-dimensional ergodic invariant tori of positive volume.
   \end{enumerate}
   Furthermore, there exists some $\de>0$ such that the Hamiltonian field associated with a function $H\in C^r(\RR^{2d})$, $r\geq 2d+1$, also has these properties provided that the $C^r(\Om)$-norm of the difference $\frac12|P|^2+W_\eta(x)- H(x,P)$ is smaller than~$\de$.
\end{proposition}

\begin{proof}
It is clear that the unperturbed Hamiltonian, that is,
\[
H_0(x,P):=\frac12(P_1^2+ x_1^2)+\sum_{j=2}^d(P_j^2+1-\cos x_j)\,,
\]
is formed by a rotor in the variables $(P_1,x_1)$ and $d-1$ pendulums in the variables $(P_j,x_j)$, $j=2,\dots,d$. It is convenient to write the rotor term using action--angle coordinates $(I,\theta)$ defined as $P_1=\sqrt{2I}\cos\theta$, $x_1=\sqrt{2I}\sin\theta$, which yields
\[
H_0(I,\theta,x_2,\cdots,x_d,P_2,\cdots,P_d)=I+\sum_{j=2}^d(P_j^2+1-\cos x_j)\,.
\]
Fixing an energy $h>2(d-1)$, the corresponding Hamiltonian vector field has two periodic orbits $\gamma_+,\gamma_-$ on the level set $\{H_0=h_0\}$ given by the periodic orbit of the rotor
\[
I(t)=I_0:=h-2(d-1)\,,\qquad \theta(t)=\theta+t\,,
\]
and the hyperbolic equilibria of the pendulums
\[
(P_j(t),x_j(t))=(0,\pm \pi)
\]
for all $j=2,\cdots,d$. It is easy to check that both $\gamma_+$ and $\gamma_-$ are hyperbolic. For each pendulum, the two aforementioned equilibria are connected by two separatrices, whose natural parametrizations (as solutions to the corresponding Hamiltonian systems) are well known to be
\begin{equation}\label{eq.separatrix}
(P_{j}^\pm(t),x_{j}^\pm(t))=\pm(2\sech t,2\arctan(\sinh t))\,,
\end{equation}
for all $j=2,\cdots,d$. These separatrices define $d$-dimensional invariant manifolds of the periodic orbits $\gamma_{\pm}$, which can be parametrized as
\[
(I_0,\theta+s_1,P_2^{\pm}(s_2),x_2^{\pm}(s_2),\cdots,P_d^{\pm}(s_d),x_d^{\pm}(s_d))
\]
with $(s_1,s_2,\cdots,s_d)\in \mathbb S^1\times \RR^{d-1}$. We then conclude that the unstable manifold $W\un(\gamma_+)$ coincides with the stable manifold $W\st(\gamma_-)$ and the stable manifold $W\st(\gamma_+)$ coincides with the unstable manifold $W\un(\gamma_-)$. We denote these connections as $\Gamma_{\pm}$, and they form what is known as a heteroclinic cycle (on the energy level $h$).

We claim that the Hamiltonian field $X_{W_\eta}$, on the energy level set
\[
H_\eta:=\frac12|P|^2+W_\eta(x)=h
\]
exhibits two hyperbolic periodic orbits with transverse homoclinic intersections, provided that $\eta$ is small enough. To prove this, by~\cite[Chapter~2.7, Corollary~1]{Palis} it is enough to show that each one of the branches $\Gamma_{\pm}$ of the aforementioned heteroclinic cycle between $\gamma_+$ and $\gamma_-$ intersects transversely when $\eta\neq0$ is small. We notice that by standard hyperbolic theory, the periodic orbits $\gamma_+$ and $\gamma_-$, and their stable and unstable manifolds, persist (up to a small deformation) when $\eta$ is small. Generically, the intersection sets of the perturbed stable and unstable manifolds are $1$-dimensional, and can be described using Melnikov's theory.

Following~\cite{DLS}, the Melnikov potential associated to the Hamiltonian $H_\eta$ and the connection $\Gamma_{\pm}$ is given by
\begin{align*}
 M_\pm(\theta,\tau_1,\cdots,\tau_{d-1})=-\frac12\sum_{j=1}^{d-1}\int_{-\infty}^\infty\Big(\sqrt{2I_0}\sin(\sigma+\theta)\mp2\arctan\sinh(\sigma+\tau_j)\Big)^2d\sigma
\\+\frac12\sum_{j=1}^{d-1}\int_0^\infty\Big(\sqrt{2I_0}\sin(\sigma+\theta)\mp\pi\Big)^2d\sigma+
\frac12\sum_{j=1}^{d-1}\int_{-\infty}^0\Big(\sqrt{2I_0}\sin(\sigma+\theta)\pm\pi\Big)^2d\sigma\,,
\end{align*}
where we have used the expression of the separatrices of the pendulums, cf. Equation~\eqref{eq.separatrix}. The convergence of these improper integrals is standard, and guaranteed by the hyperbolicity of the periodic orbits of the unperturbed system. According to Melnikov's theory, the existence of transverse intersections between the perturbation of the connections $\Gamma_{\pm}$ follows from the existence of nondegenerate critical points of the function $M_{\pm}$ in the variables $(\tau_1,\cdots,\tau_{d-1})$, for each $\theta\in[0,2\pi)$. Let us then compute the partial derivative $\partial_{\tau_j}M_{\pm}$:
\begin{align*}
 \frac{\partial M_\pm}{\partial\tau_j}=\pm2\int_{-\infty}^\infty\Big(\sqrt{2I_0}\sin(\sigma+\theta)\mp2\arctan\sinh(\sigma+\tau_j)\Big)\sech(\sigma+\tau_j)d\sigma\,.
\end{align*}
As usual, to compute this improper integral, we make a change of integration variable $s_j=\sigma+\tau_j$, which yields
\begin{align*}
\frac{\partial M_\pm}{\partial\tau_j}&=\pm2\int_{-\infty}^\infty\Big(\sqrt{2I_0}\sin(s_j+\theta-\tau_j)\mp2\arctan\sinh s_j\Big)\sech s_j\,ds_j\\
&=\pm2\pi\sqrt{2I_0}\sech\frac{\pi}{2}\,\sin(\theta-\tau_j)\,.
\end{align*}
In the second equality we have used that
\[
\int_{-\infty}^\infty \arctan\sinh s_j \sech s_j\,ds_j=\int_{-\infty}^\infty \sin s_j \sech s_j\,ds_j=0
\]
because the integrands are odd functions.

We infer that the critical points of $M_\pm$ in the variables $(\tau_1,\cdots,\tau_{d-1})$, for fixed~$\theta$, are
\[
p_k:=(\theta+\pi k_1,\dots, \theta+\pi k_{d-1})
\]
with $k\in\ZZ^{d-1}$, and they are all nondegenerate because the Hessian matrix
\[
\frac{\partial^2 M_\pm}{\partial\tau_j\partial \tau_l}\bigg|_{p_k}=\mp2\pi\sqrt{2I_0}\sech\frac{\pi}{2}\, (-1)^{k_j} \, \de_{jl}
\]
is obviously invertible. Therefore, the two branches of the heteroclinic cycle of $H_0$ intersect transversely when the perturbation parameter $\eta$ is small enough, as we wanted to show.

In summary, we have proved that, for each $K>2(d-1)$ and all values of $h\in (2(d-1),K)$, if $\eta$ is small enough, the perturbed Hamiltonian system $H_\eta$ exhibits on each energy level two hyperbolic periodic orbits with transverse homoclinic intersections. It is obvious that if $K$ is close to $2(d-1)$, these periodic orbits and homoclinic intersections are contained on the set $\Om$ defined in the statement of the lemma. A straightforward application of the Smale--Birkhoff theorem~\cite[Theorem~5.3.5]{GH} then shows the existence of a horseshoe near the destroyed heteroclinic cycle, and hence in the set~$\Om$.

To study the invariant tori of the perturbed system, we write the unperturbed Hamiltonian $H_0$ using action--angle coordinates $(I_j,\vp_j)$ for the pendulums. For this we need to consider a periodic annular domain in all the variables $(P_j,x_j)$ which is contained in the subset $\Om$ of the phase space. In these coordinates, the Hamiltonian takes the form
\[
H_0(I,I_1,\cdots,I_{d-1})=I+\sum_{j=1}^{d-1}F(I_j)
\]
for some analytic function $F$ whose explicit expression is not relevant for our purposes. We only need to know that $F$ is not an affine function (because the period of the periodic orbits of the pendulum is not constant, contrary to the case of the rotor). Now we can apply the isoenergetic KAM theorem~\cite[Section~6.3.2]{Arnold} noticing that the $(d+1)\times (d+1)$ matrix
\[
\left(
  \begin{array}{cc}
    \nabla^2H_0 & \nabla H_0 \\
    \nabla H_0 & 0 \\
  \end{array}
\right)\,,
\]
where the derivatives are taken with respect to the $d$ action variables, is nondegenerate because its determinant is given by $(-1)^d \prod_{j=1}^{d-1}F''(I_j)$, which is nonzero except for a discrete set of values of the actions $I_j$ (this follows from the fact that $F$ is analytic and not affine). This is precisely the twist type condition needed to apply the isoenergetic KAM theorem on the energy level set $H_\eta=h$ for each $h\in (2(d-1),K)$, where $K$ is the constant that has been taken before. We then infer the existence of a positive volume set of $d$-dimensional ergodic invariant tori for the Hamiltonian $H_\eta$ provided that $\eta$ is small enough,

To complete the proof of the lemma, we observe that the existence of transverse homoclinic intersections for hyperbolic periodic orbits is robust for $C^2$-small perturbations of the Hamiltonian, while the isoenergetic KAM theorem~\cite[Section~6.3.2]{Arnold} holds for $C^r$-small perturbations, with $r\geq 2d+1$. The lemma then follows.
\end{proof}

The potential $W_\eta$ constructed in Proposition~\ref{L.chaos0} is not in the space $\cF_\cQ$ because its Fourier transform $\widehat W_\eta$ (which is well defined as a tempered distribution because $W_\eta$ is polynomially bounded) is not supported on the box $\cQ$. In the following lemma we prove that $W_\eta$ can be approximated on compact sets by a function $W\in\cF_\cQ$. This is key to exploit the properties of the probability measure we introduced in Section~\ref{S.Rd}, whose support is precisely the space~$\cF_\cQ$. In the statement, $\Om$ is the bounded set introduced in Proposition~\ref{L.chaos0} and $r\geq 2d+1$.

\begin{lemma}\label{L.chaos}
There is a constant $\La_0$, such that for each $\La>\La_0$ there exists a function $W_1\in\cF_\cQ$, a constant~$c_1>0$ and some $\de>0$ such that, if $w\in C^r(\RR^d)$ is a potential with $\|w-W_1\|_{C^r(K)}<\de$, then the corresponding Hamiltonian vector field $X_w$ on~$\RR^{2d}$ has a horseshoe and a set of ergodic $d$-dimensional invariant tori of volume greater than~$c_1$ which are contained in~$\Om$, on a nonempty interval of energy levels. Here $K\subset\RR^d$ is the closure of the $x$-projection of the set $\Om\subset\RR^{2d}$.
\end{lemma}
\begin{proof}
Let $\widehat W_\eta$ be the Fourier transform of $W_\eta$, which is a tempered distribution on~$\RR^d$, and consider a smooth cutoff function $\chi:\RR^d\to[0,1]$ which is equal to~1 on the cube $[-\frac\La2,\frac\La2]^d$ and vanishes identically outside $\cQ=[-\La,\La]^d$.
Then
\[
\widehat W_1(\xi):=\widehat W_\eta(\xi)\,\chi(\xi)
\]
is a distribution on $\RR^d$ whose support is contained in $\cQ$. Accordingly, if we define the potential $W_1$ on $\RR^d$ as the Fourier transform of $\widehat W_1$, we infer that it is analytic (by Paley--Wiener) and that it belongs to the space $\cF_\cQ$ (see Remark~\ref{R.obvious}).

Since $\widehat W_1\to \widehat W_\eta$ in the sense of tempered distributions as $\La\to\infty$, standard mapping properties of the Fourier transform imply that $W_1\to W_\eta$ in $C^r(K)$ as $\La\to\infty$, for any fixed compact set $K\subset\RR^d$. Therefore, for any fixed~$K$ there exists some~$\La_0$ so that
\begin{equation}\label{eq:apro}
\|W_1-W_\eta\|_{C^r(K)}<\delta
\end{equation}
provided that $\La>\La_0$.

Taking $r\geq 2d+1$ and $\delta$ small enough, it follows from Proposition~\ref{L.chaos0} and the estimate~\eqref{eq:apro} that the Hamiltonian vector field $X_{W_1}$ has a horseshoe and a positive volume set of ergodic invariant tori in $\Om$, on a nonempty interval of energy levels. The same property holds for any potential $w$ with $\|w-W_1\|_{C^r(K)}<\de$ because of the obvious estimate
\[
\|w-W_\eta\|_{C^r(K)}\leq \|W_1-W_\eta\|_{C^r(K)}+\|w-W_1\|_{C^r(K)}\leq 2\delta\,,
\]
provided that $\delta$ is small enough. This completes the proof of the lemma.
\end{proof}

We are now ready to state and prove the main result of this section, which ensures that the number of horseshoes of the Gaussian random Hamiltonian field~$X_W$ contained in a large ball grows at least volumetrically with probability~1. For technical reasons that will be clear in the next section, it is useful to restrict our attention to tori satisfying certain uniform nondegeneracy conditions. Therefore, let us introduce the following notation:

\begin{definition}\label{D.NR}
	For each $w\in C^r(\RR^d)$ and each $R>0$, we shall denote by~$\cH_R(w)$ the number of horseshoes of the Hamiltonian vector field~$X_w$ that are contained in the set
\begin{equation}\label{cBR}
	\cB_R:= (B_R\times B_R)\cap \{h_1<H_w<h_2\}\subset \RR^{2d}\,,
\end{equation}
where $H_w$ is the Hamiltonian function associated with the potential~$w$ and $(h_1,h_2)$ is a fixed energy interval. Likewise,  we will denote by $\cT_R(w)$ the greatest number of pairwise disjoint compact invariant sets in~$\cB_R$ that contain a volume greater than $V_0>0$ of $d$-dimensional $C^r$-tori invariant under the flow of~$X_w$ and such that:
\begin{enumerate}
\item On the invariant torus, the Hamiltonian vector field $X_w$ is $C^r$-conjugate to a linear flow whose frequency vector~$\om$ satisfies a Diophantine condition
\[
|\om\cdot k|\geq \ga |k|^{-(d-1)-\tau}
\]
for all $k\in\ZZ^d\backslash\{0\}$, with the constants $\ga>\ga_1$ and $\tau\in(\tau_1,\tau_2)$, for some fixed $\ga_1,\tau_1,\tau_2>0$.

\item The invariant torus satisfies the isoenergetic twist condition~\cite{Haro} (see Hypothesis $H_4'$ in Theorem~3.6, with $n=d$), and the twist determinant is larger than a fixed positive constant~$c_1$.
\item The $C^r$ norm of the linearizing embedding $K:\mathbb T^d\to\RR^{2d}$ of item $(i)$ is bounded by a fixed constant $c_2$.
\end{enumerate}
Furthermore, we shall use the notation
\[
\Phi_R(w):= \min \big\{ \cH_R(w),\cT_R(w)\big\}\,.
\]
\end{definition}

\begin{remark}\label{R.chaos}
By Proposition~\ref{L.chaos0}, we can choose the parameters $h_1,h_2,\ga_1,\tau_1,\tau_2,c_1,c_2$ and $V_0$ in this definition so that any potential with $\|w-W_\eta\|_{C^r(B_5)}<\de$ satisfies $\Phi_5(w)\geq1$, for $r\geq 2d+1$. In fact, we observe that the Hamiltonian field $X_w$ exhibits a horseshoe and a positive $(d-1)$-volume set of invariant tori in the level set $\{H_w=h\}$ for each $h\in (h_1,h_2)$.
\end{remark}

%
%For each $w\in C^r(\RR^d)$ and each $R>0$, we shall denote by~$\Phi_R(w)$ the number of horseshoes of the vector field~$X_w$ that are contained in the product of two balls of radius~$R$, i.e., $B_R\times B_R\subset\RR^{2d}$. For the ease of notation, for each $V_0>0$, we shall also denote by~$\Phi_R(w)$ the number of pairwise disjoint compact invariant sets which contain a set of ergodic $d$-dimensional invariant tori of volume greater than $V_0$, contained in $B_R\times B_R\subset\RR^{2d}$.

The main result of this section can then be stated as follows. The constant $\La_0$ in the statement was introduced in Lemma~\ref{L.chaos}.

\begin{theorem}\label{T:chaos}
For any $\La>\La_0$, with probability~$1$,
\[
\liminf_{R\to\infty}\frac{\Phi_R(W)}{R^d}>\nu_0\,,
\]
where~$\nu_0$ is a positive constant independent of~$R$.
\end{theorem}

\begin{proof}
Let us define
\[
\cA:=\big\{ w\in C^r(\RR^d) : \cT_5(\tau_y w)\geq1\; \text{ for some } y\in\RR^d\big\}\,,
\]
where we recall that $\tau_y w(\cdot)=w(\cdot+y)$.
As $\cA$~is invariant under the group of translations $\{\tau_y: y\in\RR^d\}$ and the measure~$\mu_W$ is ergodic by Lemma~\ref{L.ergodic}, $\mu_W(\cA)$ is either~0 or~1. By Lemma~\ref{L.chaos} and Remark~\ref{R.chaos}, taking $r\geq 2d+1$, for all $\La>\La_0$ there is a function $W_1\in\cF_\cQ$ and some~$\de>0$ such that
\[
\cA':=\{w\in C^r(\RR^d): \|w-W_1\|_{C^r(B_5)}<\de\}\subset \cA\,.
\]
Lemma~\ref{L.support} then guarantees that
$\mu_W(\cA')>0$, so one then concludes that $\mu_W(\cA)\geq \mu_W(\cA')$ must be~1.

Next, let us denote by $\widetilde \cT_R(w)$ the number of connected components of the canonical projection into the first factor $B_R\subset\RR^d$ of the sets counted by $\cT_R(w)$. It is obvious that $\cT_R(w)\geq \widetilde\cT_R(w)$. Since closedness is preserved by projection, the lower bound for sets whose components are closed proved in~\cite[Lemma 6.1]{Beltrami} ensures that, for any $0<R_0<R$,
		\begin{align*}
			\frac{\cT_R(W)}{|B_R|}\geq \frac{\widetilde\cT_R(W)}{|B_R|}\geq \frac1{|B_R|}
			\int_{B_{R-R_0}}\frac{\widetilde\cT_{R_0}(\tau_yW)}{|B_{R_0}|}\, dy \geq \frac1{|B_R|}
			\int_{B_{R-R_0}}\frac{\cT_{R_0}^m(\tau_yW)}{|B_{R_0}|}\, dy\,,
		\end{align*}
		where for any large~$m>1$ we have defined the truncation
		\[
		\cT_R^m(w):= \min\{\widetilde\cT_R(w),m\}\,.
		\]

		As the truncated random variable~$\cT_R^m$ is in
		$L^1(C^r(\RR^d),\mu_W)$ for any~$m$, one can consider the limit
		$R\to\infty$ as in the proof of~\cite[Theorem 6.2]{Beltrami} and invoke
		 Lemma~\ref{L.ergodic} to conclude that
		\begin{align*}
			\liminf_{R\to\infty}\frac{\cT_R(W)}{|B_R|}\geq
			\liminf_{R\to\infty} \frac{|B_{R-R_0}|}{|B_{R}|}\Bint_{B_{R-R_0}}\frac{\cT_{R_0}^m(\tau_yW)}{|B_{R_0}|}\,
			dy = \frac{\bE \cT_{R_0}^m}{|B_{R_0}|}=: c_{R_0,m}
		\end{align*}
		$\mu_W$-almost surely, for any~$R_0$ and $m$. By definition, for $m\geq1$ one has
		\[
		\bE \cT_{R_0}^m\geq \mu_W(\cA')>0
		\]
		if $R_0\geq 5$, so we conclude that		
		 $c_{R_0,m}>0$ for any large enough~$R_0$. As the same argument applies to $\cH_R(W)$, the theorem then follows.
	\end{proof}

\section{Some convergence results}
\label{S.convergence}

In this section we show that there is a connection between the random trigonometric potentials on $\TT^d$ defined in the Introduction, and the random potential $W$ on $\RR^d$ introduced in Section~\ref{S.Rd}. Roughly speaking, the former converges to the latter when the trigonometric order $L$ tends to $\infty$.

We first introduce some notation. Given a point $q\in\TT^d$ and a constant $\La>0$, let us define the rescaled random Gaussian
function
\[
V^{L,q}(x):= V^L(q+\La x/L)\,,
\]
with $x\in\RR^d$. Comparing with~\cite{NS16,Sodin}, we have introduced an additional
constant~$\La>1$ in the scaling to obtain some additional leeway that
will be key in the proof of Theorem~\ref{T1}.

A key ingredient in the proof is the relation between the covariance
kernels of the various Gaussian random functions that we have
introduced so far. Following Nazarov and Sodin~\cite{NS16}, we will be most interested in
the covariance kernel of the rescaled
trigonometric potential~$V^{L,q}$ at a
point~$q\in\TT^d$:
\begin{align}\label{eq:def kerL}
  \ka^{L,q}(x,x')&:= \bE[V^{L,q}(x)\, V^{L,q}(x')]\,.
\end{align}

\begin{proposition}\label{P.covL}
  For any $q\in\TT^d$, the rescaled covariance kernel
  $\ka^{L,q}(x,x')$ has the following properties:
  \begin{enumerate}
  \item It is invariant under translations and independent
    of~$q$. That is, there exists a function $\varkappa^L$ such
    that
    \[
\ka^{L,q}(x,x')=\varkappa^L(x-x')\,.
    \]

    \item Given any compact set
      $K\subset\RR^d$ and any positive integer~$r$,
  \[
\varkappa^{L}(x)\to \varkappa_W(x)
  \]
  in $C^r(K)$ as $L\to\infty$.
  \end{enumerate}
\end{proposition}
\begin{proof}
Just as in~\eqref{Eb}, the Gaussian random variables that enter the definition of~$V^L$ satisfy
\[
\bE
(a_n\overline{a_{n'}})=\de_{n, n'}\,.
\]
Therefore, as $V^{L,q}$ is real valued,
\begin{align*}
	\ka^{L,q}(x,x')& = d_L^{-1}\sum_{0<|n|_\infty,|n'|_\infty\le L}\bE(a_n\overline{a_{n'}})\,e^{in\cdot (q+\La x/L)}\, e^{-in'\cdot (q+\La x'/L)}
	\\&=d_L^{-1}\sum_{0<|n|_\infty\le L} e^{i\Lambda n\cdot (x-x')/L}=\int_{\cQ}e^{i\xi\cdot(x-x')}d\rho_L(\xi)=:\varkappa^L(x-x')\,.
\end{align*}
Here we have defined the discrete probability measure
\[
{\rho}_L\coloneqq \frac{1}{d_L}\sum_{0<|n|_\infty\le L}\delta_{\La n/L}\,,
\]
whose support is contained in the cube $\cQ:=[-\La,\La]^d$.
This is the first assertion of the proposition.

To prove the second assertion, start by recalling that, as shown in~\eqref{E.KWint}, the covariance kernel of~$W$ can be written as
\[
\varkappa_W(x)= \int_{\cQ}e^{ix\cdot\xi}\, (2\La)^{-d}\,d\xi\,,
\]
where $ (2\La)^{-d}\, d\xi$ is the normalized Lebesgue measure on~$\cQ$.
The claim that $\varkappa^L \to\varkappa_W$  is then a consequence of the fact that, by standard properties of Riemann sums, the measure $\rho_L$ converges weakly to the normalized Lebesgue measure on~$\cQ$. To see this, let us write
\begin{equation}\label{E.convAA}
\pd^\al \varkappa^L(x)-\pd^\al \varkappa_W(x)= \int_{\cQ}(i\xi)^\al \, e^{i\xi\cdot x}\, \left[d\rho_L(\xi)- (2\La)^{-d}\,d\xi\right]\,.
\end{equation}
For each $x\in\RR^d$ and each multiindex~$\al$, consider the function
\[
f_{\al,x}(\xi):=(i\xi)^\al \, e^{i\xi\cdot x} \,.
\]
For each compact subset~$K\subset\RR^d$ and each positive integer~$r$, the set of functions
\[
\left\{ f_{\al,x}: x\in K,\; |\al|\leq r\right\}
\]
is bounded in $C^{r+1}(\cQ)$, and therefore precompact in $C^r(\cQ)$ by the Arzel\`a-Ascoli theorem. It then follows from the expression~\eqref{E.convAA} and from the weak convergence of measures $\rho^L \rightharpoonup (2\La)^{-d}\,d\xi$ that $\varkappa^L$ converges to $\varkappa_W$ in $C^r(K)$ as $L\to\infty$, for any $r,K$ as above. The proposition is then proven.
\end{proof}

We are now ready to establish the central result of this subsection, which asserts that, for every $q\in\TT^d$, the probability measure $\mu^{L,q}$ defined by the rescaled potential $V^{L,q}$ converges to that of~$W$ as $L\to\infty$. Intuitively speaking, this is related to the fact that the dimension of the vector space where $\mu^{L,q}$ is supported tends to infinity as $L\to\infty$, as well as to the convergence of the covariance kernel presented in Proposition~\ref{P.covL}.

\begin{lemma}\label{L.convmu}
Fix some $R>0$ and denote by~$\mu^{L,q}_R$
and~$\mu_{W,R}$, respectively, the probability measures on $C^r(B_R)$ defined by
the Gaussian random fields~$V^{L,q}$ and~$W$. Then, for any $q\in \TT^d$, the
measures~$\mu^{L,q}_R$ converge weakly to~$\mu_{W,R}$ as $L\to\infty$.
\end{lemma}
\begin{proof}
	Let us start by noting that all the finite dimensional distributions of
	the fields $V^{L,q}$ converge to those of~$W$ as $L\to\infty$. Specifically, consider any finite number of points $x^1,\dots ,
	x^n\in\RR^d$, and any
	multiindices $\al^1,\dots,\al^n$ with $|\al|\leq r$. Then it is not hard to see that the Gaussian vectors of zero
	expectation
	\[
	(\pd^{\al^1}V^{L,q}(x^1), \dots, \pd^{\al^n}V^{L,q}(x^n))\in\RR^{n}
	\]
	converge in distribution to the Gaussian vector
	\begin{equation}\label{Gvect2}
		(\pd^{\al^1}W(x^1),\dots, \pd^{\al^n}W(x^n))
	\end{equation}
	as $L\to\infty$. This follows from the fact that their probability
	density functions are completely determined by the $n\times n$
	variance matrix
	\[
	\Si^L:=
	\Big(\pd_x^{\al^l}\pd_{x'}^{\al^m}\ka^{L,q}(x,x')\big|_{(x,x')=(x^l,x^m)}\Big)_{1\leq
		l,m\leq n}\,,
	\]
	which converges to
	$\Si:=(\pd_x^{\al^l}\pd_{x'}^{\al^m}\varkappa(x,x')|_{(x,x')=(x^l,x^m)})$ as
	$L\to\infty$ by Proposition~\ref{P.covL}. The latter, of course, is
	the covariance matrix of the Gaussian vector~\eqref{Gvect2}.

	It is well known~\cite{BI} that this convergence of arbitrary Gaussian vectors
	is not enough to conclude that $\mu^{L,q}_R$ converges weakly
	to~$\mu_{W,R}$. However, notice that for
	any integer~$s\geq0$, the expectation of the $H^s$-norm of~$V^{L,q}$ is uniformly bounded:
	\begin{align*}
		\bE\|V^{L,q}\|_{H^s(B_R)}^2&=\sum_{|\al|\leq s}\bE\int_{B_R}
		|D^\al V^{L,q}(x)|^2\, dx\\
		&=\sum_{|\al|\leq s}\int_{B_R}  D_x^\al D_{x'}^\al
		\ka^{L,q}(x,x')\big|_{x'=x}\, dx\\
		&\xrightarrow[L\to\infty]{\phantom{L}}
		\sum_{|\al|\leq s}\int_{B_R}D_x^\al
		D_{x'}^\al\ka(x,x')\big|_{x'=x}\, dx<M_{s,R}\,.
	\end{align*}
	To pass to the last line, we have used Proposition~\ref{P.covL} once
	more. As the constant~$M_{s,R}$ is independent of~$L$, Sobolev's
	inequality ensures that
	\[
	\sup_L \bE\|V^{L,q}\|_{C^{r+1}(B_R)}^2\leq C\sup_L \bE \|V^{L,q}\|^2_{H^{r+2+\lfloor \frac d2\rfloor}(B_R)}<M
	\]
	for some constant~$M$ that only depends on~$R$ and~$d$. Using Chebyshev's inequality, this implies that, for all large enough~$L$ and any $\ep>0$, the probability of the event
	\[
	\cA_\ep:= \big\{ w\in C^r(B_R): \|w\|_{C^{r+1}(B_R)}^2> M/\ep\big\}
	\]
	is bounded as
	\[
	\mu^{L,q}_R(\cA_\ep)<\ep\,.
	\]
	Since the set $\big\{ w\in C^r(B_R): \|w\|_{C^{r+1}(B_R)}^2\leq M/\ep\big\}$ is a precompact subset of $C^r(B_R)$ by the Arzel\`a--Ascoli theorem, we conclude
	that the sequence
	of probability measures $\mu^{L,q}_R$ is tight.
	Therefore, a straightforward extension to jet spaces of the classical
	result about the
	convergence of probability measures on the space of continuous
	functions~\cite[Theorem 7.1]{BI}, carried out in~\cite{Wil86},
	permits to conclude that $\mu^{L,q}_R$ indeed converges weakly to~$\mu_{W,R}$ as $L\to\infty$. The lemma is then proven.
\end{proof}

\section{Proof of Theorem~\ref{T1}}
\label{S.conclusion}

We are now ready to prove Theorem~\ref{T1}.

%The basic idea is that, by the
%definition of the rescaling,
%\[
%\mu^L\big(\big\{ w\in C^k(\TT^d,\RR^d):
%N^{\mathrm{h}}_w >m\big\}\big)\geq \mu^{L,z}_R\big(\big\{ w\in C^k(B_R,\RR^d):
%N^{\mathrm{h}}_w(r)>m\big\}\big)
%\]
%provided that $r<R<L^{1/2}$: this just means that the number of horseshoes
%that~$V^L$ has in the whole torus is certainly not
%less than those that are contained in a ball centered at any given
%point $z\in\TT^d$ of radius $r/L^{1/2}<1$. The same is clearly true as
%well when one counts invariant solid tori, periodic orbits or zeros instead.

\subsubsection*{Coexistence of chaos and invariant tori with a probability that tends to~$1$}

Using the same notation as in Definition~\ref{D.NR}, we start by noticing that the map
\[
C^r(\RR^d)\ni w\mapsto\Phi_R(w)
\]
is lower semicontinuous provided that $r\geq 2d+1$. When it counts the number of horseshoes, this follows from the persistence of hyperbolic invariant sets~\cite[Theorem~5.1.2]{GH}. In the case of invariant tori, the lower semicontinuity follows from the isoenergetic KAM theorem~\cite[Theorem~3.6]{Haro} thanks to the assumption that the invariant tori we consider satisfy uniform Diophantine and twist conditions (the isoenergetic KAM theorem proved in~\cite{Haro} is for analytic Hamiltonian fields, but an analogous result holds for $C^r$ Hamiltonians, with $r\geq 2d+1$, see e.g.~\cite[Section~6.3.2]{Arnold}). Then
\[
\cA_{R,R_0}:= \{w\in C^r(B_R): \Phi_{R_0}(w)\geq 1\}
\]
is an open subset of $C^r(B_R)$ for any $0<R_0<R$. We recall that by definition, the horseshoes and invariant tori counted by $\Phi_{R_0}$ are contained in the intersection of the set $B_{R_0}\times B_{R_0}$ with the energy interval $\{h_1<H_w<h_2\}$.

As the set $\cA_{R,R_0}$ is open and the
measures $\mu^{L,q}_R$ converge weakly to~$\mu_{W,R}$ as $L\to\infty$, cf. Lemma~\ref{L.convmu}, it is well known (see e.g.~\cite[Theorem 2.1.iv]{BI}) that
 \begin{align*}
	\liminf_{L\to\infty}\mu^{L,q}_R(\cA_{R,R_0})&\geq
	\mu_{W,R}(\cA_{R,R_0})=\mu_W(\cA_{R_0})\,,
	\end{align*}
	where we have defined the event
	\[
	\cA_{R_0}:=\{w\in C^r(\RR^d): \Phi_{R_0}(w)\geq 1\}\,.
	\]
In turn, taking $\La>\La_0$, by Theorem~\ref{T:chaos}, for any $\de>0$ one can take a large enough~$R_0$ such that
\[
\mu_W(\cA_{R_0})>1-\tfrac\de2\,.
\]
Consequently, there is some $L_0$  such that
\begin{equation}\label{muLde}
\mu^{L,q}_R\big(\{w\in C^r(B_R): \Phi_{R_0}(w)=0\}\big)<\de
\end{equation}
for all $L>L_0$. As the measure $\mu^{L,q}_R$ does not depend on~$q$, $L_0$ depends on~$\de$ and~$R_0$ but not on~$q$.

Let us fix a point $q^k$ inside each cube $Q_k^N$ of the $N$-grid~$\cG^N$ and consider the balls
\[
B^{k,L}:= \left\{q\in\TT^d: |q-q^k|<\La R_0 L^{-1}\right\}
\]
centered at these points and of radius $\La R_0 L^{-1}$. Note that the closure of ${B^{k,L}}$ is contained in $Q_k^N$ for all large enough~$L$.

The key observation now is that $(x(t),P(t))$ is a trajectory of the Hamiltonian vector field defined by the random potential $V^{L,q^k}$ which is contained in~$B_{R_0}\times\RR^d$, if and only if
\begin{equation}\label{rescaling}
(q(t),p(t)):=\left(q^k + \frac\La L x(L\La^{-1}t), \; P(L\La^{-1}t) \right)
\end{equation}
is a trajectory of the Hamiltonian vector field~$X^L$ (see Equation~\eqref{pot_rando}) that is contained in $B^{k,L}\times\RR^d$; notice that this equivalence requires a rescaling of the time variable.

The transformation~\eqref{rescaling} obviously preserves the number of objects we count and the value of the energy levels:
\[
\frac12 |p|^2+V^L(q)=\frac12 |P|^2+V^L(q^k+\La x/L)=\frac12 |P|^2+V^{L,q^k}(x)\,.
\]
The symplectic forms are connected through the relation $dq\wedge dp = \frac\La L dx\wedge dP$, so the volume as computed using the $(q,p)$ variables is $(\La/L)^d$ times that computed in the $(x,P)$ variables. Therefore, for $L>L_0$, the probability that $X^L$ does not have a horseshoe and a set of invariant tori of volume $V_0(\La/L)^d>0$ in
\begin{equation}\label{QNkenergy}
(Q^N_k\times B_R) \cap \{h_1< H^L<h_2\}
\end{equation}
is at most~$\de$
by the estimate~\eqref{muLde}. Hence, the probability that~$X^L$ has a horseshoe and a positive-volume set of invariant tori (of size of order $L^{-d}$) in each of the $(2N)^d$ sets of the form~\eqref{QNkenergy} associated with the $N$-grid~$\cG^N$ is at least $1-(2N)^d\de$.
As $N$ is fixed, this probability can be taken arbitrarily close to~1 by picking~$L$ large enough.

\subsubsection*{Estimate for the expectation values}

By Theorem~\ref{T:chaos}, taking $\La>\La_0$ one can pick a large enough constant~$R_0$ so that
\[
\bE \cH_{R_0}>1\,,
\]
where $\bE$ denotes the expectation associated with the probability measure~$\mu_W$.
As discussed above, the functional $C^r(\RR^d)\ni w\mapsto \cH_{R_0}(w)$ is lower semicontinuous. Since $\mu^{L,q}_{R_0}$ converges
weakly to~$\mu_{W,R_0}$ as $L\to\infty$ by Lemma~\ref{L.convmu}, it is standard (see e.g.~\cite[Exercise 2.6]{BI}) that
	\[
	\liminf_{L\to\infty}\bE^{L,q}\cH_{R_0} \geq \bE
	\cH_{R_0}>1\,.
	\]
	Since the measure $\mu^{L,q}$ is in fact independent of~$q$, there is some $L_0$ independent of~$q$ such that $\bE^{L,q}\cH_{R_0}>1$ for all $q\in\TT^d$ and all $L>L_0$.
	
	Given a set $U\subset\TT^d\times\RR^d$ and a potential $w\in C^r(\TT^d)$, let us denote by $\cH_U(w)$ the number of horseshoes of the vector field~$X_w$ that are contained in the set
	\begin{equation}\label{setU}
	U\cap \{h_1<H_w<h_2\}\,.
	\end{equation}
	In view of the scaling property described in Equation~\eqref{rescaling}, now it suffices to pick, for each $L>L_0$, a collection of $\lfloor c_* L^d\rfloor$ pairwise disjoint balls $B^l$ centered at points $Q^l\in\TT^d$ and of radius $\La R_0 L^{-1}$ to conclude that the expected number of horseshoes that the vector field~$X^L$ has on the set $\{h_1<H^L<h_2\}$ is at least
	\begin{align*}
	\bE^L\cH_{\TT^d\times\RR^d} &\geq \sum_{l=1}^{\lfloor c_* L^d\rfloor} \bE^L \cH_{B^l\times B_{R_0}} \geq  \sum_{l=1}^{\lfloor c_* L^d\rfloor} \bE^{L,Q^l} \cH_{R_0}>  \nu_* L^d\,.
	\end{align*}
	Here $\nu_*$ is a positive constant related to the packing of spheres constant, which also depends on $\La$ and $R_0$.
	
	Now let us denote by $\cV_U(w)$ the inner volume of the set of Diophantine invariant tori of~$X_w$ that are contained in~\eqref{setU}. Arguing as before, and noting that the volume of the set of invariant tori picks a factor of $(\La/L)^d$ by the scaling properties of the transformation~\eqref{rescaling}, one infers that
	\begin{align*}
	\bE^L\cV_{\TT^d\times\RR^d} &\geq  \sum_{l=1}^{\lfloor c_* L^d\rfloor} \bE^L \cV_{B^l\times B_{R_0}} \geq  V_0(\La/L)^d\sum_{l=1}^{\lfloor c_* L^d\rfloor} \bE^{L,Q^l} \cT_{R_0}>  \nu_*
	\end{align*}
	for some positive constant $\nu_*$.  The theorem is then proven.

\subsubsection*{Proof of Corollary~\ref{cor}}

For arbitrary $d\geq2$, it is known that a Hamiltonian system that is completely integrable with nondegenerate first integrals, has zero topological entropy~\cite{Pater}, and hence it cannot exhibit any horseshoes. When $d=2$, Moser proved~\cite[Corollary 4.8.5]{GH} the stronger result that any Hamiltonian system with a horseshoe does not admit a global analytic second first integral (independent of the Hamiltonian). The corollary then follows.\hfill\qed

\section{Random Hamiltonian systems on an arbitrary compact
  manifold}
\label{S.manifolds}

In this section we will sketch how the results of this paper can be extended, with minor modifications, to the case of natural Hamiltonian systems defined on an arbitrary compact $d$-dimensional manifold~$M$.

Let us start by picking a smooth Riemannian metric~$g$ on~$M$ and defining the Hamiltonian function on the cotangent bundle $T^*M$ associated with a smooth potential $V\in C^\infty(M)$ as
\[
H_V(q,p):=\frac12|p|^2_g+V\,.
\]
To define a random potential on~$M$ analogous to~\eqref{VL}, one can use the eigenvalues and eigenfunctions of the Laplacian on~$M$ as a substitute of the frequencies and exponentials employed in~\eqref{VL}. Specifically, let $\{u_j\}_{j=0}^\infty$ be an $L^2(M)$-orthonormal basis of real-valued eigenfunctions of the Laplacian on~$M$, which satisfies the equation
\[
\De_M u_j+\la_j u_j=0
\]
on $M$ with constants
\[
0=\la_0<\la_1\leq \la_2\leq\cdots
\]
It is well known that $u_0$ is constant and that $\la_j$ tends to infinity as $c_M j^{2/d}+o(j^{2/d})$, where $c_M$ is a positive constant depending on the manifold.

One can then define an ensemble of random potentials on~$M$ analogous to~\eqref{VL}. Specifically, for each $J\geq1$ we introduce the random Gaussian potential
\[
\tV^J(q):= \frac1{\sqrt{J}}\sum_{j=1}^J a_j \, u_j(q)\,,
\]
where $a_j$ are independent standard random variables. When $(M,g)$ is the flat torus~$\TT^d$, we recover the situation studied in the rest of the paper because, in this case, $\tV^J \sim V^L$ when $J\sim L^d$. Note that the potential is a $C^\infty$ function because it is a finite linear combination of Laplace eigenfunctions, which are smooth.

Letting $X^J$ be the Hamiltonian vector field on $T^*M$ associated with the random potential~$\tV^J$, one can prove the following result:

\begin{theorem}\label{T2}
The probability that the random Hamiltonian vector field~$X^J$ is neither integrable with nondegenerate first integrals nor ergodic tends to~$1$ as $J$~tends to infinity.
More precisely, let $\{U^k:1\leq k\leq N\}\subset M$ be a collection of $N$ pairwise disjoint domains with smooth boundary. The probability that $X^J$ has a horseshoe and a set of ergodic invariant tori of positive volume (of order $J^{-1}$) contained in each of the sets $T^*U^k$ tends to~$1$ as $J\to\infty$. In both cases, the horseshoes and invariant tori we count are on the energy levels of a fixed nonempty interval $(h_1,h_2)$ which does not depend on $J$.

Furthermore, the expected number of horseshoes of~$X^J$ in these fixed energy levels tends to infinity as $J\to\infty$:
\[
\liminf_{J\to\infty} \bE^J\cH_{T^*M}=\infty\,.
\]
\end{theorem}

The proof goes exactly as in the case of the flat torus, with three minor variations. Firstly, given any point $q\in M$, the rescaled potentials must be defined as
\[
\tV^{J,q}(x):= \tV^J(\exp_q(\La\, \la_J^{-1/2}x))\,,
\]
where $\exp_q$ is the exponential map at~$q$. (Equivalently, one could replace $\la_J$ in this formula by the asymptotic approximation $c_MJ^{2/d}$.) Second, Proposition~\ref{P.covL} must be replaced by the convergence results for the covariance kernel of $\tV^{J,q}$ as a consequence of H\"ormander's local Weyl law, see e.g.~\cite{Canzani}. Third, the measures $\mu^{J,q}$ are no longer independent of~$q$, in general. This is why we cannot easily derive the estimate for the expected number of horseshoes and invariant tori, which appears in Theorem~\ref{T1} but not in~\ref{T2}. However, this does not affect the proof of the first half of Theorem~\ref{T1} (i.e., the first part of Section~\ref{S.conclusion}) because the number~$N$ is fixed throughout. In fact, for each fixed~$N$ we can proceed as in the second part of Section~\ref{S.conclusion} to write:
\begin{align*}
\bE^J \cH_{T^*M}\geq \sum_{l=1}^N \bE^{J,Q_l}\cH_{R_0}>N\eta\,,
\end{align*} 
for all large enough $J$ and some $\eta>0$, thus implying that
\[
\liminf_{J\to\infty} \bE^J\cH_{T^*M}=\infty\,.
\]
The nontrivial dependence of $\mu^{J,q}$ with $q$ prevents us from estimating the growth of $\bE^J\cH_{T^*M}$ in terms of $J$.

\section*{Acknowledgements}

This work has received funding from the European Research Council (ERC) under the European Union's Horizon 2020 research and innovation programme through the grant agreement~862342 (A.E.). It is partially supported by the grants CEX2019-000904-S, RED2018-102650-T and PID2019-106715GB GB-C21 (D.P.-S., A.R.) funded by MCIN/AEI/10.13039/501100011033. A.R. is also a postgraduate fellow of the City Council of Madrid at the Residencia de Estudiantes (2020--2022).

\bibliographystyle{amsplain}

\end{document}